\documentclass[12pt]{amsart}
	
	\usepackage{amssymb,amsmath,amscd,amsthm}
	\usepackage{graphicx,psfrag,epsfig}
	\usepackage[active]{srcltx}
	\usepackage{thm-restate}
	\usepackage{color,soul}
	\usepackage{xcolor}
	\usepackage{enumitem}

	\newtheorem{theorem}{Theorem}[section]
	
	\newtheorem{corollary}[theorem]{Corollary}
	
	\newtheorem{lemma}[theorem]{Lemma}
	\newtheorem{proposition}[theorem]{Proposition}
	\newtheorem{defn}[theorem]{Definition}

	\numberwithin{equation}{section}
	
	\newcommand{\Bound}{\partial \Omega}
	
	\date{}

\begin{document}
		\author{Ivan Pombo}
		\address{ Department of Mathematics\\
			University of Aveiro\\
			Campus Universit\'ario de Santagio\\
			3810-193 Aveiro, Portugal}
		\email{ivanpombo@ua.pt}
		\thanks{Department of Mathematics, Aveiro University, Aveiro 3810, Portugal.  This work was supported by Portuguese funds through CIDMA -
			Center for Research and Development in Mathematics and Applications and the Portuguese Foundation for Science and Technology 
			(``FCT--Funda\c{c}\~{a}o para a Ci\^{e}ncia e a Tecnologia''), within project UID/MAT/04106/2019}
		
		\title{Reconstruction from Boundary Measurements complex conductivities}
		
		\begin{abstract}
			In this paper we show that following Nachman's method we can still reconstruct complex conductivities in $C^{1,1}$ from its Dirichlet-to-Neumann map in three and higher dimensions. For such, we analyze all of the results in \cite{Nachman} and pinpoint what really needs to be shown for complex conductivities. Moreover, we also obtain low frequency estimates for $C^{1,1}$-boundaries following the approach in  \cite{Knudsen_3D}. As far as we aware, this is the first reconstruction procedure for complex conductivities, even though the proof follows trivially by extending some of Nachman's theorems to the complex case.
		\end{abstract}
		
		\maketitle
		
		
		\section{Introduction}
		
		In Electrical Impedance Tomography (EIT) we determine the interior impedance inside a bounded domain $\Omega$ by applying  alternating electrical currents and measuring the corresponding voltages at the boundary $\Bound$, or vice-versa. Impedance is the inverse of admittance which is defined through $\gamma = \sigma + i\omega\epsilon$, where $\omega$ is the angular frequency, $\sigma,\, \epsilon$ are the electrical conductivity and permittivity of materials inside $\Omega$, respectively. 
		
		Our working assumptions are:
		\begin{align}
			\label{condition_1}&\gamma \in L^{\infty}(\bar{\Omega}) \; \text{ and isotropic,}\quad \sigma \geq c > 0, \quad \epsilon\geq 0, \quad \omega\in\mathbb{R^+},\\
			\label{condition_2}& \Omega \;\text{ is a }\; C^{\,1,1}\,-\text{domain in } \mathbb{R}^n,\, n\geq 3 			
		\end{align} 
		
		In applications, most data acquisition systems and corresponding algorithms just focus on computing the conductivity $\sigma$. However, in certain applications it is highly valuable to also obtain the permittivity from boundary measurements. It brings extra knowledge to the table and allows the distinction of new clinical conditions that are not possible with just the conductivity. 
		An example, it is the ability to distinguish between pneumothorax and hyperinflation. Both scenarios correspond to regions of low resistivity, which implies high conductivity, but the pneumothorax has zero permittivity while the hyperinflation corresponds to low yet positive permittivity. Other application will be in multi-frequency EIT since the properties $\sigma$ and $\epsilon$ actually vary with the applied angular frequency $\omega$, while in the real case the frequency is somewhat discarded. 
		
		Mathematically, the direct problem is modeled by assuming we have a voltage $f\in H^{1/2}(\Bound)$ set at the boundary and the objective is to find the electrical potential $u$ which is the unique solution in $H^1(\Omega)$ of:
		
		\begin{align}\label{EIT}
			\begin{cases}
			\nabla\cdot(\gamma\nabla u) = 0,\, \text{ in } \Omega \\
			\left.u\right|_{\Bound} = f
			\end{cases}
		\end{align}
		
		 Uniqueness in $H^1(\Omega)$ holds from the fact that $\text{Re }\gamma > 0$, which implies by the weak formulation that $0$ is not a Dirichlet eigenvalue of the operator $\nabla\cdot(\gamma\nabla u)$ in $\Omega$. 
		 
		 Formally, from each voltage $f\in H^{1/2}(\Bound)$ and each corresponding electrical potential $u\in H^1(\Omega)$ we can define the electrical current measured at the boundary by  $\gamma\frac{\partial u}{\partial\nu}$. To be precise, for $\gamma \in L^{\infty}(\Omega)$ by:		
		\begin{align}\label{DtN_map}
			\Lambda_{\gamma}: H^{1/2}(\Bound)&\rightarrow H^{-1/2}(\Bound),\\
			\nonumber f\quad&\mapsto \quad 
			\langle \Lambda_{\gamma}f,g\rangle = \int_{\Omega} \gamma\nabla u\cdot\nabla v\,dx	
		\end{align}
		where $v\in H^1(\Omega)$ has trace $g$ in $\Bound$.
		
		In 1980 A.P. Calderón \cite{Calderon} was the first to pose the mathematical problem of whether the conductivity $\sigma\in L^{\infty}(\Omega)$ can be uniquely determined by boundary measurements, $\Lambda_{\gamma}$, and if so how to reconstruct it. He showed that the linearized problem at constant conductivities has a unique solution. This problem is well-known in the literature as the Calderón problem or inverse conductivity problem. In medical imaging the problem is know by EIT. 

		After the initial work of Calderón there were many extensions to global uniqueness results. In \cite{Sylvester_Uhlmann}, Sylvester and Uhlmann used ideas of scattering theory, namely the exponential growing solutions of Faddeev \cite{Faddeev} to obtain global uniqueness in dimensions $n\geq 3$ for smooth conductivities. Using this foundations the uniqueness for lesser regular conductivities was further generalized for dimensions $n\geq 3$ in the works of (\cite{Nach_Syl_Uhl}, \cite{Alessandrini}, \cite{Nachman}, \cite{Chanillo}, \cite{Brown}, \cite{PPU}). Currently, the best know result is due to Brown and Torres \cite{Brown_Torres} for conductivities $\gamma\in W^{3/2, p}(\Omega)$ with $p>2n$. The reconstruction procedure for $n\geq 3$ was obtained in both \cite{Nachman} and \cite{Novikov} independently. In dimension two the problem seems to be of a different nature and tools of complex analysis were used to established uniqueness. Nachman \cite{Nachman2D} obtained uniqueness and a reconstruction method for conductivities with two-derivatives. The uniqueness result was soon extend for once-differentiable conductivities in \cite{Brown_Uhlmann} and a corresponding reconstruction method was obtained in \cite{Knudsen_Tamasan}. In 2006, Astala and Päivärinta gave a positive answer for $\sigma\in L^{\infty}(\Omega),\, \sigma\geq c>0$ in \cite{Astala_Paivarinta}.
		
		The first extension to admittances, and here forward also designated by complex-conductivities, was made in \cite{Francini}. In this paper, Francini extended the work of Brown and Uhlmann \cite{Brown_Uhlmann} in two-dimensions by proving uniqueness for small angular frequencies and $\gamma \in W^{2,\infty}$. Afterwards, Bukgheim  influential paper \cite{Bukgheim} proved the general result in two-dimensions for conductivities in $W^{2,\infty}$. He reduced the  (\ref{EIT}) to a Schrödinger equation and shows uniqueness through the stationary phase method (based on is work many extensions followed \cite{Imanuvilov_Yamamoto}, \cite{Novikov_Santacesaria}, \cite{Astala_Faraco_Rodgers}). Recently, by mixing techniques of \cite{Brown_Uhlmann} and \cite{Bukgheim}, Lakshtanov, Tejero and Vainberg obtained in \cite{Lakshtanov_Tejero_Vainberg} uniqueness for once-differentiable complex-condutivities. In \cite{Pombo}, the author followed up their work to show that it is possible to reconstruct complex-conductivity with a jump at least in a certain set of points.\\ 
		
		As far as we know, there is not many results in higher-dimensions presented in literature for complex-conductivities. However, as stated in \cite{Borcea}, it is possible to extend the uniqueness results obtained in \cite{Sylvester_Uhlmann} and  \cite{Nach_Syl_Uhl}, thus uniqueness holds for twice-differentiable complex-conductivities. In terms of reconstruction, there is no reference for the extension of Nachman D-bar procedure \cite{Nachman} to the complex case or of any other direct reconstruction methods. However, in \cite{Hamilton_3D} Nachman's reconstruction is used to compute complex conductivities from boundary measurements, which are in fact promising results that led us to formalize the arguments.
		\\
		
		In this paper, we show that Nachman's reconstruction method holds for complex conductivities. We follow closely the proof in \cite{Nachman} and show where the proof requires other results including complex-conductivities. Moreover, we show that the complex-conductivity can be obtained from low-frequency asymptotics through the exponential growing solutions, following the results obtained in \cite{Knudsen_3D}. To keep in Nachman's setting our approach just requires that the boundary is $C^{1,1}$. For completeness sake, we present the most important proofs of Nachman and pinpoint how to generalize them for complex-conductivities.
		
		\section{Uniqueness of Schrödinger Inverse Problem}
		
		We start from scratch and present the uniqueness result in \cite{Nach_Syl_Uhl} for complex-valued potentials in $L^{\infty}(\bar{\Omega})$. In their work, there is no mention and need of $q$ being real, therefore we present their proof in its entire.
		\\
		
		As is well-known from the literature, we reduce our problem to an analogous one with the Schrödinger operator $(-\Delta+q)$. If $0$ is not a Dirichlet eigenvalue of this operator in $\Omega$, then $\Lambda_q$ is well-defined from $H^{1/2}(\Bound)$ to $H^{-1/2}(\Bound)$ and formally is given by
		$$\Lambda_qf = \left.\frac{\partial w}{\partial\nu}\right|_{\Bound},$$ for $w$ the unique solution of $(-\Delta+q)w=0, \text{ in } \Omega$ and $\left. w\right|_{\Bound}=f$. The corresponding inverse problem will be to determine $q$ from the boundary measurements $\Lambda_q$.
		\\
		 
		Let $u\in H^1(\Omega)$ be the unique solution of (\ref{EIT}) with trace $f\in H^{1/2}(\Bound)$. Then the substitution $u=\gamma^{-1/2}w$ yields
		
		\begin{align}\label{Schrodinger_Inside}
			-\Delta w+qw=0, \text{ in } \Omega, \text{ with } q=\frac{\Delta \gamma^{1/2}}{\gamma^{1/2}}.
		\end{align}
		
		Notice that if $\gamma\in C^{1,1}(\Omega)$ and $\sigma\geq c>0$ then $\gamma^{1/2}$ is well-defined and twice-differentiable. Therefore, $q$ is well-defined and in $L^{\infty}(\Omega)$. Furthermore, the above conditions on the complex-conductivity imply that $0$ is not a Dirichlet eigenvalue of $\nabla\cdot(\gamma\nabla )$ and therefore $0$ is also not a Dirichlet eigenvalue of the Schrödinger operator $(-\Delta+q)$ for $q$ given as in (\ref{Schrodinger_Inside}).
		
		Now, we can focus on studying the Schrödinger equation. By extending $q$ to zero outside the domain we can study solutions of
		\begin{align}\label{Sch_eq1}
			-\Delta w+qw=0, \;\text{ in } \mathbb{R}^n
		\end{align}
		
		which behave like $$w=e^{ix\cdot\zeta}\left(1+\psi(x,\zeta)\right), \text{ for } \zeta \in \mathbb{C}^n,\; \zeta\cdot\zeta=0.$$ In Calderón paper \cite{Calderon} he already uses the family of exponential harmonic functions, $e^{ix\cdot\zeta}$ in its proof, but was Sylvester and Uhlmann \cite{Sylvester_Uhlmann} that first used this type of solutions to dispense the requirement of $\sigma$ to be close to a constant. Applying the substitution to (\ref{Sch_eq1}), it follows that $\psi$ must satisfy:
		\begin{align}\label{Faddeev_eq1}
			-\Delta\psi-2i\zeta\cdot\nabla\psi+q\psi = -q
		\end{align} 
		
		It is well-know that for $\zeta\in\mathbb{C}^n,\, \zeta\cdot\zeta=0$ the operator $(-\Delta-2i\zeta\cdot\nabla)$ has the fundamental solution:
		\begin{align}\label{g_zeta}
			g_{\zeta}(x) = \frac{1}{(2\pi)^n}\int_{\mathbb{R}^n} \frac{e^{ix\cdot\xi}}{|\xi|^2+2\zeta\cdot\xi}\,d\xi,
		\end{align}
		
		which leads to $G_{\zeta}(x) = e^{ix\cdot\zeta}g_{\zeta}(x)$ being a fundamental solution of the Laplace operator in $\mathbb{R}^n$: $$-\Delta G_{\zeta}(x)=\delta(x).$$ A quick remark is that this operator differs from the classical fundamental solution by an Harmonic function $H_{\zeta}$.
		
		From this, the appropriate solutions to (\ref{Sch_eq1}) can be obtained by solving the integral equation
		\begin{align}\label{Schr_int_eq1}
			w(x,\zeta)= e^{ix\cdot\zeta} - \int G_{\zeta}(x-y)q(y)w(y,\zeta)\,dy.
		\end{align}
		with $\psi$ solving
		\begin{align}\label{Schr_int_eq2}
			\psi+g_{\zeta}\ast(q\psi) = -g_{\zeta}\ast(q)
		\end{align}
		
		The study of these integral equations follows by a weighted $L^2$ estimate for $g_{\zeta}$ obtained in \cite{Sylvester_Uhlmann}, which guarantees unique solvability of (\ref{Schr_int_eq1}) for $|\zeta|$ large, even for complex conductivities.
		
		Let $\langle x\rangle = (1+|x|^2)^{1/2}$. We define the weighted $L^2$-space for $\delta\in\mathbb{R}$ as
		
		$$L^2_{\delta}(\mathbb{R}^n):=\left\{f: \|f\|_{\delta}:= \|\langle x\rangle^{\delta}f\|_{L^2(\mathbb{R}^n)}<\infty\,\right\}.$$
		
		\begin{proposition}\label{conv_est}
			For all $\zeta\in\mathbb{C}^n$ with $\zeta\cdot\zeta=0$ and $|\zeta|\geq a$ the operator of convolution with $g_\zeta$ satisfies
			\begin{align}
			\left\|g_{\zeta}\ast f\right\|_{\delta-1}\leq \frac{c(\delta, a)}{|\zeta|}\|f\|_{\delta},\quad \text{ for }\; 0<\delta<1
			\end{align}
			Moreover, let $H^2_{\delta}(\Omega):=\left\{f: D^{\alpha}f\in L^2_{-\delta}(\mathbb{R}^n), 0\leq|\alpha|\leq 2\right\}$ be the weighted Sobolev space with norm $$\|f\|_{2,\delta} = \left(\sum_{|\alpha|\leq 2} \|D^{\alpha}f\|_{\delta}^2\right)^{1/2}$$.
			
			Then, for 
			any $\zeta\in \mathbb{C}^n$ with $\zeta\cdot\zeta=0$ it holds for $\delta\in(1/2,1)$ that
			\begin{align*}
			\|g_{\zeta}\ast w\|_{2,-\delta} \leq c(\delta,\zeta)\|w\|_{2,\delta}
			\end{align*}
			Furthermore, under the definition
			$$\textbf{G}_{\zeta}w(x)=\int_{\Omega} G_{\zeta}(x-y)w(y)\,dy$$ it holds that $$\|\textbf{G}_{\zeta}w\|_{H^2(\Omega)} \leq c(\zeta, \Omega) \|w\|_{L^2(\Omega)}.$$			
		\end{proposition}
	
	\begin{proof}
		The proof of the above estimates can be found in \cite{Sylvester_Uhlmann} and \cite{Nachman}.\\
		\end{proof}

		For the uniqueness proof our interest resides in studying the exponential growing solutions given through the equation (\ref{Schr_int_eq1}):
		
		\begin{corollary}\label{uniqueness}
			Let $0<\delta<1$ and $q \in L^\infty(\Omega)$ complex-valued and extended to zero outside $\Omega$. 
			
			Then there exists an $R>0$ such that for all $\zeta\in\mathbb{C}^n$ with $\zeta\cdot\zeta=0$ and $|\zeta|>R$ the integral equation (\ref{Schr_int_eq1}) is uniquely solvable with $e^{-ix\cdot\zeta}w(x,\zeta)-1 \in L^2_{\delta-1}(\mathbb{R}^n)$.
			
			Furthermore, it holds
			\begin{align}\label{w_estimate}
			\|e^{-ix\cdot\zeta}w(x,\zeta)-1\|_{\delta-1} \leq \frac{\tilde{c}(R,\delta)}{|\zeta|}\|q\|_{\delta}.
			\end{align}
		\end{corollary}
	
		\begin{proof}
			Let $M_q\phi = q\phi$, i.e., the operator of multiplication with $q$. We show that for $q\in L^{\infty}(\mathbb{R}^n)$ with compact support,  $M_q: L^2_{\delta-1}(\mathbb{R}^n)\rightarrow L^2_{\delta}(\mathbb{R}^n)$ is a bounded operator.
			
			Let $f \in L^2_{\delta-1}(\mathbb{R}^n)$.
			\begin{align*}
			\|M_qf\|_{\delta} &= \left[\int_{\mathbb{R}^n} (1+|x|^2)^{\delta}|q(x)f(x)|^2\, dx\right]^{1/2} \\
			&= \left[\int_{\mathbb{R}^n} \left(1+|x|^2\right)|q(x)|^2 \left(1+|x|^2\right)^{\delta-1}|f(x)|^2\,dx\right]^{1/2} \\
			& \leq \|\langle x\rangle q \|_{\infty}\|f\|_{\delta-1}.\\
			\end{align*}
			
			We define the operator $A_{\zeta}=C_{\zeta}M_q$, where $C_{\zeta}$ is the convolution with $g_{\zeta}$, that is:
			
			\begin{align}\label{K_zeta}
			A_{\zeta}f(x) = \int_{\mathbb{R}^n} g_{\zeta}(x-y)q(y)f(y)\,dy = C_{\zeta}M_qf
			\end{align}
			
			By proposition \ref{conv_est} for $|\zeta|\geq R$ we obtain:
			
			\begin{align*}
			\|A_{\zeta}f\|_{\delta-1} = \|C_{\zeta}M_qf\|_{\delta-1} \leq \frac{c(\delta, R)}{|\zeta|}\|M_qf\|_{\delta}\leq \frac{c(\delta, R)}{|\zeta|}\left\|\langle x\rangle q\right\|_{\infty}\|f\|_{\delta-1}
 			\end{align*}
			
			Therefore, $K_{\zeta}$ is bounded in $L^2_{\delta-1}(\mathbb{R}^n)$. Further, if we set  $|\zeta|>R:=c(\delta, R)\left\|\langle x\rangle q\right\|_{\infty}$ then $A_{\zeta}$ is a contraction and $I+A_{\zeta}$ is invertible.
			
			Since $q\in L^{\infty}$ and as compact support then it is in $L^2_{\delta}$ and therefore the right-hand side of (\ref{Schr_int_eq2}) is in $L^2_{\delta-1}$. Hence, the unique solution to (\ref{Faddeev_eq1}) is given by:

			$$\psi(x,\zeta) = -\left[I + A_{\zeta}\right]^{-1}\left(g_{\zeta}\ast q\right).$$
			
			From here, we already know that $w = e^{ix\cdot\zeta}\left(1 - \left[I + A_{\zeta}\right]^{-1}\left(g_{\zeta}\ast q\right) \right)$ solves the integral equation (\ref{Schr_int_eq1}). Furthermore, the estimate (\ref{w_estimate}) easily follows from $[I+A_{\zeta}]^{-1}$ being bounded in $L^2_{\delta-1}$, proposition \ref{conv_est} and $g_{\zeta}\ast q\in L^2_{\delta-1}$.
			
			Now, let us suppose that there exist two solutions $w_1, w_2$ of (\ref{Schr_int_eq1}) such that $\phi_j = e^{-ix\cdot\zeta}w_j-1 \in L^2_{\delta-1}$. Then, their difference is also in $L^2_{\delta-1}$ and both fulfill the equation $[I+A_{\zeta}]\phi_j = -g_{\zeta}\ast q$.
			
			This implies: $[I+A_{\zeta}]\left(e^{-ix\cdot\zeta}(w_1-w_2)\right) = 0 \Rightarrow w_1\equiv w_2$ by the invertibility of $I+A_{\zeta}$ in $L^2_{\delta-1}$.
			
			Hence, uniqueness of the integral equation (\ref{Schr_int_eq1}) for exponential growing solutions follows.\\
		\end{proof}
	
		We designate the values $\zeta$ for which the solution does not exist or is not unique as exceptional points. To be precise		
		
		\begin{defn}
			Let $q\in L^{\infty}(\Omega)$ complex-valued and extended to zero outside $\Omega$. 
			
			Let $\zeta\in\mathcal{V}:=\{\zeta\in\mathbb{C}^n\setminus\{0\}: \zeta\cdot\zeta=0\}$. Then we call $\zeta\in\mathcal{V}$ an exceptional point for $q$ if there is no unique exponential growing solution of $(-\Delta+q)w=0$ in $\mathbb{R}^n$, that is, there is no unique solution of the type:
			\begin{align*}
				w(x,\zeta):=e^{ix\cdot\zeta}\left(1+\mu(x,\zeta)\right), \text{ with } \mu\in L^2_{\delta-1}(\mathbb{R}^n),\, 0<\delta<1.
			\end{align*}
		\end{defn}
		
		The uniqueness proof of this section given by \cite{Nach_Syl_Uhl} and Nachman's reconstruction method \cite{Nachman}, only require large non-exceptional points $\zeta$. However, in this sense the D-bar method is very unstable and would be desirable to mimic the theory in two-dimensions, where we are able to reconstruct $\gamma$ from small values of non-exceptional points $\zeta$.
		
		\begin{lemma}
			Let $q_1,\, q_2 \in L^{\infty}(\Omega)$ and extended to zero outside $\Omega$. Let $\zeta\in\mathcal{V}$ a non-exceptional point for $q_1,\,q_2$.
			
			Suppose that $\Lambda_{q_1}=\Lambda_{q_2}$ and $w_1,\, w_2$ are the unique solutions of $(-\Delta+q_j)w_j = 0$ in $\mathbb{R}^n$ of the form $e^{ix\cdot\zeta}\left(1+\mu_j\right)$.
			
			Then, $$w_1 = w_2,\quad \text{ in } \mathbb{R}^n\setminus\Omega.$$
		\end{lemma}
		
		\begin{proof}
			Let $v \in H^1(\Omega)$ be the unique solution of 
			\begin{align*}
			-\Delta v + q_2\, v &= 0, \quad \text{ in } \Omega \\
			\left.v\right|_{\Bound} &= \left.w_1\right|_{\Bound}
			\end{align*} 
			
			Define, 
			\begin{align*}
			h = \left\{
			\begin{array}{l}
			v,\quad \text{ in } \Omega \\
			w_1, \;\;\text{ in } \mathbb{R}^n\setminus\Omega
			\end{array}
			\right.
			\end{align*}
			
			Since, $\Lambda_{q_1} = \Lambda_{q_2}$ it holds that $\Lambda_{q_1}\left.w_1\right|_{\Bound} = \Lambda_{q_2}\left.w_1\right|_{\Bound} \Rightarrow \frac{\partial w_1}{\partial\nu} = \frac{\partial v}{\partial\nu}$. This implies that $h$ is continuous over $\Bound$, as well as, $\frac{\partial h}{\partial\nu}$. 
			
			Therefore, $h$ solves $-\Delta h + q_2 \,h = 0$ in $ \mathbb{R}^n$ and has the appropriate asymptotics since $w_1$ has them. By the uniqueness theorem it follows that $h=w_2$ and thus $w_1=w_2$ in $\mathbb{R}^n\setminus\Omega$.\\
		\end{proof}
		
		Now the uniqueness follows:
		
		\begin{theorem}
			Let $q_1, \,q_2 \in L^{\infty}(\Omega)$ extended to zero outside $\Omega$. Suppose that $0$ is not a Dirichlet eigenvalue of $-\Delta+q_j$, $j=1,\,2$ on $\Omega$.\\
			
			\centering{If $\Lambda_{q_1}=\Lambda_{q_2}$, then $q_1=q_2$.}
		\end{theorem}
		
		\begin{proof}
			
			Let $k\in\mathbb{R}^n$ be fixed and for $m, s\in\mathbb{R}^n$ we set
			\begin{align*}
			\zeta = \frac{1}{2}\left(\left(k+s\right) + im\right)\, \text{ and }\,\tilde{\zeta} = \frac{1}{2}\left(\left(k-s\right) - im\right)
			\end{align*} 
			with $k\cdot s = k\cdot m = s\cdot m = 0$ and $|k|^2+|s|^2 = |m|^2$.
			
			The $\zeta,\,\tilde{\zeta}$ are in $\mathbb{C}^n$ and fulfill the condition $\zeta\cdot\zeta=0$. Hence, taking $s, \,m$ large enough we obtain solutions $w_j$ of the integral equation (\ref{Schr_int_eq1}) for their respective potentials for $\tilde{\zeta}$. By Green's identity it holds:
			
			\begin{align*}
			\int_{\Omega} e^{ix\cdot\zeta}q_j(x)w_j(x)\,dx = \int_{\Omega} e^{ix\cdot\zeta}\Delta w_j(x) - w_j\Delta e^{ix\cdot\zeta}\,dx = \int_{\Bound} e^{ix\cdot\zeta}\frac{\partial w_j}{\partial \nu} - w_j\left(\nu\cdot i\zeta\right)e^{ix\cdot\zeta}\,d\sigma(x). 
			\end{align*}
			
			By hypothesis and the previous lemma $\Lambda_{q_1}=\Lambda_{q_2} \Rightarrow \left.w_1\right|_{\Bound}=\left.w_2\right|_{\Bound}$. Therefore, it also holds that $\left.\frac{\partial w_1}{\partial\nu}\right|_{\Bound}=\left.\frac{\partial w_2}{\partial\nu}\right|_{\Bound}$, since $w_j$ solve the interior problem $(-\Delta+q_j)w_j=0$.
			
			Hence, the right-hand side of the integral above is equal for both $q_j$ and assuming the asymptotics of $w_j$ w.r.t. $\tilde{\zeta}$ it follows
			\begin{align*}
			\int_{\Omega} e^{ix\cdot\zeta}\left(q_1w_1 - q_2w_2\right)\, dx = 0 \Leftrightarrow \int_{\Omega} e^{ix\cdot(\zeta+\tilde{\zeta})}\left(q_1-q_2\right)\,dx = \int_{\Omega} e^{ix\cdot(\zeta+\tilde{\zeta})}\left(q_1\psi_1-q_2\psi_2 \right),dx  
			\end{align*}
			
			Using $\zeta+\tilde{\zeta}=k$ and taking modulus we obtain by Cauchy-Schwarz inequality and Corollary \ref{uniqueness}
			
			\begin{align*}
			\left|\int_{\Omega} e^{ix\cdot k}\left(q_1-q_2\right)\, dx\right| \leq \sum_{j=1}^{2} \int_{\Omega} \left|q_j\psi_j\right| \leq \sum_{j=1}^2 \|q_j\|_{1-\delta}\|\psi_j\|_{\delta-1} 				\leq \sum_{j=1}^2 \frac{C}{|\tilde{\zeta}|}\|q_j\|_{1-\delta}\|q_j\|_{\delta}
			\end{align*}
			
			Since, $\tilde{\zeta}$ was arbitrarily depending on $s$, we can take the limit as $|s|\rightarrow \infty$. This implies that the left-hand side equals to zero for each fixed $k \in \mathbb{R}^n$. Given that the proof holds for all $k$ we have
			
			$$\int_{\Omega} e^{ix\cdot k}(q_1 - q_2)\,dx = 0,\quad \forall k \in \mathbb{R}^n$$
			
			Therefore, by Fourier inversion theorem we obtain $q_1=q_2$ in $\Omega$.\\
		\end{proof}

		Notice two things: first we do not require more assumptions for $q$ being complex; second the following uniqueness proof only works for $n\geq 3$ due to the required choice of $\zeta,\,\tilde{\zeta}$.
		
		Hence, uniqueness is extended for complex-potentials in $L^{\infty}(\Omega)$ with $0$ not a Dirichlet eigenvalue of $(-\Delta+q)$. To extend uniqueness for admittivities $\gamma\in C^{1,1}(\bar{\Omega})$ it is still necessary to establish a relation between $\Lambda_{\gamma}$ and $\Lambda_q$. We will present this in a later section.
		
		\section{Preliminaries for Reconstruction}
		
		In this section we present the necessary results to follow Nachman's approach \cite{Nachman} and establish a boundary integral equation and obtain a scattering transform for complex-conductivities from where we obtain the Fourier transform of the $q$. Moreover, we also present some estimates to prove invertibility of the boundary integral equation for small values of $\zeta$, following \cite{Knudsen_3D}.
		\\
		
		Analogously to the classical single and double layer potentials we define the respective operators for $G_{\zeta}$.
		
		The single layer operator is defined as
		\begin{align*}
		S_{\zeta}f(x) = \int_{\Bound} G_{\zeta}(x-y)f(y)\,ds(y)
		\end{align*}
		and the double layer as	
		\begin{align*}
		D_{\zeta}f(x) = \int_{\Bound} \frac{\partial G_{\zeta}}{\partial\nu}(x-y)f(y)\,ds(y).
		\end{align*}
		
	Moreover, taking the trace of double layer potential it holds
	\begin{align*}
	B_{\zeta}f(x) := \text{ p.v. } \int_{\Bound} \frac{\partial G_{\zeta}}{\partial\nu}(x-y)f(y)\,ds(y), \text{ for } x\in\Bound.
	\end{align*}
	
	Since the singularity of $G_{\zeta}$ for $x$ near $y$ is the same as $G_0$, it is locally integrable on $\Bound$ and the trace of $S_{\zeta}$ is still "itself".
	\\
	
	We state here the properties that Nachman established and are essential for the later proofs.		
		\begin{proposition}
			Let $\Omega$ be a bounded $C^{1,1}$-domain in $\mathbb{R}^n,\, n\geq 3$.
			\begin{enumerate}[label=(\roman*)]
				\item For $0\leq s\leq 1$ \begin{align}\label{single_layer}
				\left\|S_{\zeta}f\right\|_{H^{s+1}(\Bound)} \leq c(\zeta, s)\|f\|_{H^s(\Bound)}.
				\end{align}
				\item For $0\leq s\leq \frac{3}{2}$ we have that $B_{\zeta}$ is bounded in $H^s(\Bound)$.
			\end{enumerate}
		\end{proposition}
		
		Let $\rho_0$ be a number large enough so that $\bar{\Omega}\subset \{x: |x|<\rho_0\}$. For any $\rho>\rho_0$ we define $\Omega_{\rho}' = \{x: x\not\in\bar{\Omega}, \, |x|<\rho\}$.
		
		\begin{lemma}
			If $f\in H^{1/2}(\Bound)$, the function $\phi=S_{\zeta}f$ has the following properties:
			\begin{enumerate}[label=(\roman*)]\label{singlelayer}
				\item $\Delta\phi = 0$ in $\mathbb{R}^n\setminus\Bound$.
				\item $\phi\in H^2(\Omega)$ and $\phi\in H^2(\Omega_{\rho}')$ for any $\rho>\rho_0$.
				\item $\phi$ satisfies an analogue to the Sommerfeld radiation condition. For almost every $x$ it holds:
				\begin{align}\label{Sommerfeld}
				\lim_{\rho\rightarrow \infty} \int_{|y|=\rho} \left[G_{\zeta}(x-y)\frac{\partial\phi}{\partial\nu(y)} - \phi(y)\frac{\partial G_{\zeta}}{\partial\nu(y)}(x-y)\right]\,ds(y) = 0.
				\end{align}
				
				In fact, for $\rho>\rho_0$ the above identity holds for $|x|<\rho$  even without taking the limit.
				
				\item Let $B_{\zeta}^{\dag}$ denote the operator on the boundary
				\begin{align}
				B_{\zeta}^{\dag}f(x) =\textnormal{ p.v. } \int_{\Bound} \frac{\partial G_{\zeta}}{\partial\nu(x)}(x-y)f(y)\,ds(y).
				\end{align}
				
				It follows that the (nontangential) limits $\partial\phi/\partial\nu_+, \, \partial\phi/\partial\nu_-$ of the normal derivative of $\phi$ as the boundary is approached from the outside and inside $\Omega$, respectively, are given by:
				
				\begin{align}\label{normal_derivatives_single}
				\frac{\partial \phi}{\partial\nu_{\pm}} = \mp \frac{1}{2}f(x) + B_{\zeta}^{\dag}f(x),\quad \text{ for almost every } x\in\Bound.
				\end{align} 
				
				\item The boundary values $\phi_{+},\, \phi_{-}$ of $\phi$ from outside and inside of $\Omega$, respectively, are identical as elements of $H^{3/2}(\Bound)$ and agree with the trace of the single layer potential $S_{\zeta}f$.\\
			\end{enumerate}
		\end{lemma}
		
		\begin{lemma}\label{doublelayer}
			If $f\in H^{3/2}(\Bound)$ the function $\psi = D_{\zeta}f$ defined in $\mathbb{R}^n\setminus\Bound$ has the properties (i), (ii) and (iii) of the Lemma \ref{singlelayer}.
			
			Moreover, the non-tangential limits $\psi_+,\,\psi_-$ of $\psi$ as we approach the boundary from outside and inside of $\Omega$, respectively, exist and satisfy
			\begin{align}
			\psi_{\pm}(x) = \pm\frac{1}{2}f(x) + B_{\zeta}f(x),\text{ for almost every } x\in\Bound.
			\end{align}\\
		\end{lemma}
	
		\begin{lemma}\label{harmonicestimate}
			Let $\Omega$ be a bounded $C^{1,1}$-domain in $\mathbb{R}^n,\, n\geq 3$.	The Faddeev fundamental solution $G_{\zeta}$ can be given through the decomposition $$G_{\zeta}(x)=G_{0}(x)+H_{\zeta}(x),$$ where $G_0$ is the classical fundamental solution and $H_{\zeta}$ is an harmonic function.
			
			Moreover, the single and double layer operators have a similar decomposition and, for our own convenience, we present here the case for the single layer. For $f\in H^{1/2}(\Bound)$ we have
			$$S_{\zeta}f(x) = S_{0}f(x) + \int_{\Bound} H_{\zeta}(x-y)f(y)\,ds(y) =: S_0f(x)+\mathcal{H}_{\zeta}f(x).$$
			
			Further, it holds:
			\begin{align*}
				\|\mathcal{H}\|_{\mathcal{L}(H^{1/2}(\Bound), H^{3/2}(\Bound)} \leq C|\zeta|^{n-2},
			\end{align*}
			
			where the constant $C$ only depends on the domain.
			
		\end{lemma}
	
		\begin{proof}
			See \cite{Knudsen_3D} for further details.\\
		\end{proof}
		
		To complete the preliminaries, we cast our focus onto the solvability of interior problem. This result will be the principal difference from Nachman's work and what some of the readers may have noticed already. 
		
		\begin{proposition}\label{uniqueness_Schro}
			Let $\Omega$ be a bounded $C^{1,1}-$domain in $\mathbb{R}^n,\, n\geq 3$. Suppose that $q\in L^{\infty}(\bar{\Omega})$ is complex-valued and that $0$ is not a Dirichlet eigenvalue of $(-\Delta+q)$ in $\Omega$.
			Then for every $f\in H^{3/2}(\Bound)$ there is a unique $w \in H^{2}(\Omega)$ such that
			\begin{align}\label{interiorDS}
				\begin{cases}
					\left(-\Delta+q\right)w=0 \text{ in } \Omega \\
					\left.w\right|_{\Bound}= f.
				\end{cases}		
			\end{align}
			
			The solution operator is defined by $P_qf := w$ and has the mapping property $$P_q: H^{3/2}(\Bound) \rightarrow H^{2}(\Omega).$$
			
			Moreover, the Dirichlet-to-Neumann map operator has the mapping property:
			$$\Lambda_q: H^{3/2}(\Bound)\rightarrow H^{1/2}(\Bound).$$ 
		\end{proposition}
	
		\begin{proof}
			The proof follows by the studying first the Laplacian and showing that multiplication by $q$ is a compact operator from $H^2(\Omega)$ to $L^2(\Omega)$.
			
			First, let $$P_0: H^2(\Omega)\rightarrow L^2(\Omega)\times H^{3/2}(\Bound), \; u\mapsto \left(\, - \Delta u,\, \text{tr } u\right).$$
			
			By the definition of $H^2(\Omega)$ and the trace properties on this space and $C^{1,1}-$domains the operator $P_0$ is linear and bounded. By Theorem 9.15. of \cite{Gilbarg_Trudinger} and under our conditions on the domain, there always exists a unique solution in $H^2(\Omega)$ of
			\begin{align*}
				\begin{cases}
				-&\Delta u = f \\
				&\left.u\right|_{\Bound} = g
				\end{cases}.
			\end{align*}
			
			Therefore, the operator $P_0$ is bijective and invertible. Thus, it is Fredholm of index zero.
			\\
			
			Analogously, we define the operator:
			
			$$P_q: H^2(\Omega) \rightarrow L^2(\Bound)\times H^{3/2}(\Bound), \, u \mapsto \left(\left[-\Delta+q\right]u,\; \text{tr } u\right).$$
			
			Then, the difference operator $P_q-P_0$ defined in the same spaces maps $u$ to $(qu,\; 0)$. 
			
			Since, the embedding $H^2(\Omega)\hookrightarrow L^2(\Omega)$ is compact, then it immediately follows multiplication by $q\in L^\infty(\Omega)$ is a compact operator. Hence, by definition $P_q-P_0$ is a compact operator in these spaces.
			\\
			
			Since, $P_q = P_0 + (P_q-P_0)$ is the sum of a Fredholm of index zero and a compact operator, it is a Fredholm operator of index zero. Thus, to show invertibility we prove that $\text{ker } P_q = \{\}$. Let $w \in \text{ker } P_q$. By definition this implies $w$ is a solution in $H^2(\Omega)$ of
			
			\begin{align*}
				\begin{cases}
				-\Delta w+qw=0 \\
				\left.w\right|_{\Bound} = 0
				\end{cases},
			\end{align*}
			
			but due to the assumption of $0$ not being a Dirichlet eigenvalue of $(-\Delta+q)$ in $\Omega$ it follows that $w\equiv0$.
			\\			
		\end{proof}
	
		\textbf{Remark: }Our main assumption on the admittivity is $\gamma \in C^{1,1}(\bar{\Omega})$, thus it is in $H^2(\Omega).$ Therefore, for potentials $q$ given by the complex-conductivity it holds:
		
		\begin{corollary}
			Let $\Omega$ be a bounded $C^{1,1}-$domain in $\mathbb{R}^n,\, n\geq 3$. For $\gamma\in C^{1,1}(\bar{\Omega})$ such that  $\text{Re }\gamma\geq c>0$. 
			
			Then $q\in L^{\infty}(\Omega)$ given by $q=\Delta(\gamma^{1/2})/\gamma^{1/2}$ is well-defined.
			
			Then the unique solution $w\in H^2(\Omega)$ of 
			\begin{align}
				\begin{cases}
				-\Delta w+qw=0\\
				\left.w\right|_{\Bound}=\gamma^{1/2}
				\end{cases}
			\end{align}
			
			is $w\equiv \gamma^{1/2}$.
		\end{corollary}

		\section{Boundary Integral Equation}
		
		The properties of the previous section allows us to establish a one-to-one correspondence between the solution of a boundary integral equation and of the following exterior problem:
		
		\begin{align}\label{exterior_prob}
			&(i)\;\Delta\psi = 0,\text{ in } \Omega':=\mathbb{R}^n\setminus\bar{\Omega} \\	
		\nonumber	&(ii)\;\psi\in H^2(\Omega_{\rho}'),\text{ for any } \rho>\rho_0, \\
		\nonumber	&(iii)\;\psi(x,\zeta)-e^{ix\cdot\zeta} \text{ satisfies } (\ref{Sommerfeld}),\\
		\nonumber	&(iv)\;\frac{\partial\psi}{\partial\nu_+} = \Lambda_q\psi\text{ on } \Bound.\\\nonumber
		\end{align}
		
		In this section, we assume that $\Omega$ is a bounded $C^{1,1}$-domain in $\mathbb{R}^n$, $n\geq 3$ and $q\in L^{\infty}(\Omega)$ is a complex-potential for which $0$ is not a Dirichlet eigenvalue.
				
		\begin{lemma}\label{Bitch2}
			Let $\zeta\in\mathcal{V}$.
			
			\begin{enumerate}[label=(\alph*)]\label{BIeq}
				\item Suppose $\psi$ solves the exterior problem (\ref{exterior_prob}). Then its trace $f_{\zeta}=\psi_+=\left.\psi\right|_{\Bound}$ solves the boundary integral equation:
				\begin{align}\label{boundaryInteq}
				f_{\zeta} = e^{ix\cdot\zeta} - \left[S_{\zeta}\Lambda_q - B_{\zeta} - \frac{1}{2}I\right]f_{\zeta}.
				\end{align}
				
				\item Conversely, suppose $f_{\zeta}\in H^{3/2}(\Bound)$ solves (\ref{boundaryInteq}). Then the function $\psi(x,\zeta)$ defined for $x\in\Omega'$ by
				\begin{align}
				\psi(x,\zeta) = e^{ix\cdot\zeta} - \left(S_{\zeta}\Lambda_q - D_{\zeta}\right)f_{\zeta}(x)
				\end{align}
				solves the above exterior problem under all conditions. Furthermore, $\left.\psi\right|_{\Bound} = f_{\zeta}$. 
			\end{enumerate}
		\end{lemma}
		
		\begin{proof}
			a) Assume $\psi$ solves (\ref{exterior_prob}). We apply Green's identity to $G_{\zeta}$ and $\psi$ in $\Omega_{\rho}'$, $\rho>\rho_0$. It holds:
			
			\begin{align}
			\left(\int_{|y|=\rho}-\int_{\Bound}\right) \left[G_{\zeta}(x-y)\frac{\partial\psi}{\partial\nu_+} - \psi_+(y,\zeta)\frac{\partial G_{\zeta}}{\partial\nu_+(y)}(x-y)\right],ds(y) \\
			\nonumber = \int_{\Omega_{\rho}'} \left[G_{\zeta}(x-y)\Delta\psi(y,\zeta) - \psi(y,\zeta)\Delta_yG_{\zeta}(x-y)\right]\, dy 
			\end{align}
			
			Since, $\psi$ is harmonic on $\Omega_{\rho}'$ and $G_{\zeta}$ is the fundamental solution of $-\Delta$ we obtain for $x\in\Omega_{\rho}'$:
			
			\begin{align}
			\psi(x, \zeta) = &\int_{|y|=\rho} \left[G_{\zeta}(x-y)\frac{\partial\left(\psi-e^{iy\cdot\zeta}\right)}{\partial\nu} - \left(\psi - e^{iy\cdot \zeta}\right)\frac{\partial G_{\zeta}}{\partial\nu_+(y)}(x-y)\right]\,ds(y) 
			\\ \nonumber &+ \int_{|y|=\rho} \left[G_{\zeta}(x-y)\frac{\partial e^{iy\cdot\zeta}}{\partial\nu} - e^{iy\cdot\zeta}\frac{\partial G_{\zeta}}{\partial\nu(y)}(x-y)\right]\,ds(y)
			\\ \nonumber &-\int_{\Bound} \left[G_{\zeta}(x-y)\frac{\partial\psi}{\partial\nu_+}\,ds(y) - \int_{\Bound}\psi_+(y,\zeta)\frac{\partial G_{\zeta}}{\partial\nu(y)}(x-y)\right]\,ds(y)
			\end{align}
			
			By hypothesis (\ref{exterior_prob}-iii), the first integral vanishes. The function $e^{iy\cdot\zeta}$ is harmonic and a re-application of Green's identity to the second integral on $|y|<\rho$ equals $e^{ix\cdot\zeta}$. 
			Finally, due to (\ref{exterior_prob}-iv) the last integral is $\left[S_{\zeta}\Lambda_q-D_{\zeta}\right]\psi$.
			
			Then the function $\psi$ fulfills for $x\in\Omega'$ the identity: $$\psi(x,\zeta) = e^{ix\cdot\zeta} - \left[S_{\zeta}\Lambda_q - D_{\zeta}\right]f_{\zeta}.$$
			
			Taking the non-tangential limit to the boundary from the outside we obtain by Lemma \ref{singlelayer} and \ref{doublelayer}			
			\begin{align*}
			f_{\zeta}(x) = e^{ix\cdot\zeta} - \left[S_{\zeta}\Lambda_q - B_{\zeta} - \frac{1}{2}I\right]f_{\zeta}(x)
			\end{align*}
						
			b) Conversely, suppose $f_{\zeta}\in H^{3/2}(\Bound)$ solves the boundary integral equation (\ref{boundaryInteq}). Define a function $\psi$ in $\Omega'$ by
			
			\begin{align}\label{psi_1}
			\psi(x,\zeta) = e^{ix\cdot\zeta} - \left[S_{\zeta}\Lambda_q - D_{\zeta}\right]f_{\zeta}(x).
			\end{align}
			\\
			
			We show that this $\psi$ solves the exterior problem (\ref{exterior_prob}) from properties of the single and double layer (Lemma \ref{singlelayer} and \ref{doublelayer}). 
			
			It is immediate to see that $\psi$ fulfills the property i) of (\ref{exterior_prob}), since for $\zeta\cdot\zeta=0$ the exponential $e^{ix\cdot\zeta}$ is harmonic, and $S_{\zeta}\Lambda_qf_{\zeta}$, $D_{\zeta}f_{\zeta}$ are harmonic in $\Omega'$ by the above mentioned lemmas.	Moreover, it holds that $S_{\zeta}\Lambda_qf_{\zeta},\, D_{\zeta}f_{\zeta} \in H^2(\Omega_{\rho}'), \,\rho>\rho_0$ and further the identity (\ref{Sommerfeld}) also holds. Hence, the property ii) and iii) of the exterior problem follow.
			\\
			
			To show the last property, we approach the boundary $\Bound$ non-tangentially from the outside and we obtain, as in part a),
			$$\left.\psi\right|_{\Bound} = e^{ix\cdot\zeta} - \left[S_{\zeta}\Lambda_q - B_{\zeta} - \frac{1}{2}I\right]f_{\zeta}.$$
			
			By virtue of $f_{\zeta}$ fulfilling the boundary integral equation the right-hand side equals $f_{\zeta}$ and therefore $\left.\psi\right|_{\Bound} = f_{\zeta}$. From this and the first three properties of (\ref{exterior_prob}), that we already showed $\psi$ fulfills, we can obtain analogously to part a) 
			
			\begin{align}\label{psi_2}
			\psi(x,\zeta) = e^{ix\cdot\zeta} - S_{\zeta}\left(\frac{\partial\psi}{\partial\nu_+}\right) + D_{\zeta}f_{\zeta},\quad \text{ for } x\in\Omega'.
			\end{align}
			
			Subtracting both formulations of $\psi$, (\ref{psi_2}) and (\ref{psi_1}), the following equality holds throughout $\Omega'$
			
			\begin{align}\label{S_zeta}
			S_{\zeta}\left[\Lambda_qf_{\zeta} - \frac{\partial\psi}{\partial\nu_+}\right] = 0
			\end{align}
			
			By taking traces from the outside, it actually holds on the boundary $\Bound$. We are reminded that $S_{\zeta}\left[\Lambda_qf_{\zeta} - \frac{\partial\psi}{\partial\nu_+}\right]$ is harmonic in $\mathbb{R}^n\setminus\Bound$ and since the trace is $0$ on $\Bound$ uniqueness of the interior problem for $q\equiv 0$ implies that the equality (\ref{S_zeta}) holds everywhere. Then, its normal derivatives will be zero and subtracting them on $\Bound$ with the help of (\ref{normal_derivatives_single}) we obtain
			
			\begin{align}
			\left[\Lambda_q-\partial\psi/\partial\nu_+\right] = \frac{\partial S_{\zeta}\left[\Lambda_q-\partial\psi/\partial\nu_+\right]}{\partial\nu_-} - \frac{\partial S_{\zeta}\left[\Lambda_q-\partial\psi/\partial\nu_+\right]}{\partial\nu_+} = 0.
			\end{align}
			
			Thus the last property of the exterior problem follows.\\
		\end{proof}

		Furthermore, we are able to obtain a relation between the exterior problem and the solutions of integral equation (\ref{Schr_int_eq1}).
		
		\begin{lemma}\label{Bitch}
			Let $\zeta\in\mathcal{V}$. Then:
			\begin{enumerate}[label=(\alph*)]
				\item Suppose $\psi\in L^2_{\textnormal{loc}}(\mathbb{R}^n)$ is a solution of $\psi(x,\zeta) = e^{ix\cdot\zeta} - \int_{\mathbb{R}^n} G_{\zeta}(x-y)q(y)\psi(y,\zeta)$. 
				\\
				
				Then the restriction of $\psi$ to $\Omega'$ solves the exterior problem (\ref{exterior_prob}) and fulfills the respective properties i)-iv). \\
				
				\item Conversely, if $\psi$ solves the exterior problem (\ref{exterior_prob}), there is a unique solution $\tilde{\psi} \in L^2_{\textnormal{loc}}(\mathbb{R}^n)$ of the integral equation (\ref{Schr_int_eq1}), such that $\tilde{\psi} = \psi$ in $\Omega'$.
			\end{enumerate}
		\end{lemma}
		
		\begin{proof}
			
			a) From the proposition \ref{conv_est} it follows $\psi\in H^2_{\textnormal{loc}}(\mathbb{R}^n)$, which immediately implies property ii) of the exterior problem. Moreover, in $\mathbb{R}^n$ it holds  $(-\Delta+q)\psi=0$, thus due to $q\equiv 0$ on $\Omega'$ the property i) holds, i.e.,  $-\Delta\psi=0$ in $\Omega'$.
			
			Applying Green identity on $|y|<\rho$:
			
			\begin{align*}
			\int_{|y|=\rho} &\left[G_{\zeta}(x-y)\frac{\partial\psi}{\partial\nu(y)} - \psi(y,\zeta)\frac{\partial G_{\zeta}}{\partial\nu(y)}(x-y)\right]\,ds(y) \\
			& = \int_{|y|<\rho} G_{\zeta}(x-y)q(y)\psi(y,\zeta)\, dy + \psi(x,\zeta), \text{ for a.e. } x \text{ with } |x|<\rho.
			\end{align*} 
			
			Now, we can choose $\rho$ large in order to contain the supp of $q$. Since $\psi$ solves integral equation this means that the right-hand side equals $e^{ix\cdot\zeta}$. Moreover, we already showed that 
			$$e^{ix\cdot\zeta} = \int_{|y|=\rho} \left[G_{\zeta}(x-y)\frac{\partial e^{iy\cdot\zeta}}{\partial\nu(y)} - e^{iy\cdot\zeta}\frac{\partial G_{\zeta}}{\partial\nu(y)}(x-y)\right]\,ds(y).$$
			
			Then passing the exponential to the right-hand side, we obtain:
			
			\begin{align*}
			\int_{|y|=\rho} \left[G_{\zeta}(x-y)\frac{\partial\left(\psi-e^{iy\cdot\zeta}\right)}{\partial\nu(y)} - \left(\psi(y,\zeta)-e^{iy\cdot\zeta}\right)\frac{\partial G_{\zeta}}{\partial\nu(y)}(x-y)\right]\,ds(y) = 0, \text{ for all} \rho>\rho_0,
			\end{align*}
			
			thus property iii) follows by taking the limit as $\rho\rightarrow \infty$.
			\\
			
			Immediately, we can see that $\Lambda_q\psi_- = \frac{\partial\psi}{\partial\nu_-}$ and since $\psi\in H^2$ in a two-sided neighborhood of $\Bound$ it holds that $\psi_-=\psi_+$ and $\frac{\partial\psi}{\partial\nu_{-} }=\frac{\partial\psi}{\partial\nu_{+}}$. This leads to $\psi$ fulfilling the iv) property. \\
			
			Therefore, the restriction of $\psi$ to $\Omega'$ solves the exterior problem (\ref{exterior_prob}).\\
			
			b) Suppose $\psi$ defined in $\Omega'$ solves the exterior problem (\ref{exterior_prob}). Set $\tilde{\psi}$ by $\tilde{\psi}=P_q\psi_+$ in $\Omega$ and $\tilde{\psi}=\psi$ in $\Omega'$.
			
			Then on $\Bound$, $\tilde{\psi}_- = (P_q\psi_+)=\psi_+=\tilde{\psi}_+$ and $\frac{\partial\tilde{\psi}}{\partial\nu_-} = \Lambda_q\psi_+ = \frac{\partial\psi}{\partial\nu_+} = \frac{\partial\tilde{\psi}}{\partial\nu_+}$ by iv).			
			Thus $\tilde{\psi}$ solves $(-\Delta+q)\tilde{\psi}=0$ on $\mathbb{R}^n$. Applying Green's formula in $|y|<\rho$
			\begin{align*}
			\int_{|y|=\rho} &\left[G_{\zeta}(x-y)\frac{\partial\psi}{\partial\nu(y)} - \psi(y,\zeta)\frac{\partial G_{\zeta}}{\partial\nu(y)}(x-y)\right]\,ds(y) \\
			& = \int_{|y|<\rho} G_{\zeta}(x-y)q(y)\tilde{\psi}(y,\zeta)\, dy + \tilde{\psi}(x,\zeta), \text{ for a.e. } x \text{ with } |x|<\rho,
			\end{align*}
			
			by letting $\rho\rightarrow \infty$ the radiation condition iii) implies that the left-hand side is $e^{ix\cdot\zeta}$. 
			Thus $\tilde{\psi}$ verifies the desired integral equation in $\mathbb{R}^n$.
			
			To finalize we prove that this extension is unique. 		
			Suppose that we have two extensions $\tilde{\psi}^1, \tilde{\psi}^2 \in L^2_{\text{loc}}(\mathbb{R}^n)$ of $\psi$ which agree in $\Omega'$ and solve the integral equation everywhere. As in part a), we see that $ \tilde{\psi}^1, \tilde{\psi}^2 \in H^2_{\text{loc}}(\mathbb{R}^n)$ and $(-\Delta+q)\tilde{\psi}^j=0$ in $\mathbb{R}^n$ for $j=1,\,2$. Hence, they are in $H^2$ on a two-sided neighborhood of $\Bound$. This implies that $\tilde{\psi}^j_+=\tilde{\psi}^j_-$, for $j=1,2$, which promptly leads too $\tilde{\psi}^1_-=\tilde{\psi}^2_-$ since they agree on $\Omega'$. Now, from the uniqueness of the interior problem it follows that $\tilde{\psi}^1=\tilde{\psi}^2$.		
			
		\end{proof}
		
			\textbf{Remark:}
			The two previous lemmas establish that a solution of the boundary integral equation is equivalent to a exponential growing solution of the Schrödinger equation in $\mathbb{R}^n$. The interesting remark is that there was no explicit requirement of $\zeta$ being large. Hence, by showing that the boundary integral equation is uniquely solvable for small values of $\zeta$ we guarantee the existence of exponential growing solutions for these $\zeta$.\\
			
			Keeping this in mind, we focus now on solvability of the boundary integral equation:			
		
		\begin{proposition}\label{InvertibilityBIE}
			Let $\Omega$ be a bounded $C^{1,1}$-domain in $\mathbb{R}^n,\,n\geq 3$. Let $q$ be a complex-valued potential in $L^{\infty}(\Omega)$ and suppose that $0$ is not Dirichlet eigenvalue of $-\Delta+q$ in $\Omega$.
			
			 We define $K_{\zeta}=S_{\zeta}\Lambda_q-B_{\zeta}-\frac{1}{2}I$ and for any $\zeta\in\mathcal{V}$ it holds:
			
			\begin{enumerate}[label=(\alph*)]
				\item The operators $K_0,\,K_{\zeta}$ are compact on $H^{3/2}(\Bound)$. 
				
				\item If $\textnormal{Re }q\geq0$, then $I+K_{0}$ is invertible in $H^{3/2}(\Bound)$.
				
				\item If $\textnormal{Re }q\geq 0$ there exists an $\epsilon>0$ with $|\zeta|<\epsilon$ for which the operator $I+K_{\zeta}$ is invertible in $H^{3/2}(\Omega)$.
				
				\item There exists an $R>0$ such that for all $|\zeta|>R$ the operator $I+K_{\zeta}$ is invertible in $H^{3/2}(\Bound)$.\\
			\end{enumerate}
		\end{proposition}
		
		\begin{proof}
			Part a) follows by a compactness embedding. 
						
			Let $f\in H^{3/2}(\Bound)$ and set $w=P_qf$ as the solution of interior Dirichlet problem (\ref{interiorDS}).
			Let $x\in\Omega$ we use the Green's formula to obtain:
			
			\begin{align*}
			\int_{\Omega} G_{\zeta}(x-y)\Delta w(y)\,dy + w(x) = \left[S_{\zeta}\Lambda_q-D_{\zeta}\right]f(x) \\
			\Leftrightarrow \int_{\Omega} G_{\zeta}(x-y)q(y)P_qf(y)\, dy + w(x) = \left[S_{\zeta}\Lambda_q - D_{\zeta}\right]f(x)
			\end{align*}
			
			By letting $x$ approach the boundary non-tangentially from the inside we obtain:
			
			\begin{align*}
			&\text{tr}\, \left(G_{\zeta}\ast(qP_qf)\right) +f(x) = S_{\zeta}\Lambda_qf(x) - \left[-\frac{1}{2}f(x) + B_{\zeta}f(x)\right]\\
			&\Rightarrow \left[S_{\zeta}\Lambda_q-B_{\zeta}-\frac{1}{2}I\right]f = \text{tr}\, \left(G_{\zeta}\ast(qP_qf)\right)
			\end{align*}
			
			Hence, our desired operator follows the above factorization.
			The following mapping properties hold:
			\begin{itemize}
				\item $P_q: H^{3/2}(\Bound)\rightarrow H^2(\Omega)$;
				\item $\imath: H^2(\Omega)\rightarrow L^2(\Omega)$ is a compact embedding;
				\item $M_q: L^2(\Omega)\rightarrow L^2(\Omega)$;
				\item $\textbf{G}_{\zeta}: L^2(\Omega) \rightarrow H^2(\Omega)$ convolution with $G_{\zeta}$, which we prove up next.
				\item $\text{tr}: H^2(\Omega)\rightarrow H^{3/2}(\Bound)$
			\end{itemize}
			
			Hence, compactness of the embedding implies compactness of the desired operator.\\
			
			b) Let $\zeta=0$. In this case $G_0$ is the classical fundamental solution and the corresponding operators are the classical ones.\\
			 
			By part a), we already know that $S_{0}\Lambda_q-B_{0}-\frac{1}{2}I$ is compact on $H^{3/2}(\Bound)$. Then $I+K_0 = \left[\frac{1}{2}I+S_{0}\Lambda_q-B_{0}\right]$ is Fredholm of index zero on $H^{3/2}(\Bound)$. Therefore, it is enough to show injectivity.
			
			Let $h\in H^{3/2}(\Bound)$ such that $\left[\frac{1}{2}I+S_{0}\Lambda_q-B_{0}\right]h=0$.			Define $w=-S_{0}\Lambda_qh+D_{0}h$. Then $w$ is harmonic in $\mathbb{R}^n$, $w\in H^2(\Omega)$ and $w\in H^2(\Omega_{\rho}')$ by Lemma \ref{singlelayer} and \ref{doublelayer}. Moreover, approaching the boundary non-tangentially by the inside we obtain
			$$w_- = -S_{0}\Lambda_qh + \left(-\frac{1}{2}h +B_{0}h\right) = -\left[\frac{1}{2}I+S_{0}\Lambda_q-B_{0}\right]h=0.$$
			
			Since, the problem $-\Delta w=0,\; \left.w\right|_{\Bound}=0$ is uniquely solvable in $H^2(\Omega)$ it follows that $w\equiv 0$ in $\Omega$ and thus $\frac{\partial w}{\partial\nu_{-}}=0$ on $\Bound$.
			
			By noticing the jump relations for the single and double layer operator (check \cite{McLean}), we can deduce that:
			\begin{align*}
				[w]=w_+-w_-=w_+ = [D_{0}h]= h,\text{ and }
				\left[\frac{\partial w}{\partial \nu}\right]=\frac{\partial w}{\partial\nu_+}=-\left[\frac{\partial}{\partial \nu}S_{0}\Lambda_qh\right]= \Lambda_qh
			\end{align*}
			
			Now, by proposition \ref{uniqueness_Schro} there is a unique solution $u\in H^2(\Omega)$ of
			\begin{align*}
				\begin{cases}
					\left(-\Delta+q\right)u=0 \\
					\left. u\right|_{\Bound} = h
				\end{cases}	
			\end{align*}
			
			such that $\Lambda_qh=\left.\frac{\partial u}{\partial\nu_-}\right|_{\Bound}$.
			\\
			
			We set
			\begin{align*}
				v=\begin{cases}
				u, \text{ in } \;\Omega\\
				w, \text{ in } \;\mathbb{R}^n\setminus\Omega
				\end{cases}
			\end{align*}
			
			 and see that $u_-=w_+=h$ and $\frac{\partial u}{\partial\nu_-}=\frac{\partial w}{\partial\nu_+}=\Lambda_q$, thus it holds that $v$ and $\frac{\partial v}{\partial\nu}$ are continuous over the boundary $\Bound$. Therefore $v\in H^2(B_{\rho}(0)),\, \rho>0$ and it solves $-\Delta v+qv=0$ in $\mathbb{R}^n$, since $q\equiv 0,$ in $\mathbb{R}^n\setminus\Omega$.
			\\
			
			Let $\chi_{\rho}\in C^{\infty}_c(\mathbb{R}^n)$ such that $\chi\equiv 1$ in $B_{\rho-\epsilon}(0)$ and $\chi\equiv 0$ in $\rho-\epsilon < |x| <\rho$, for $\epsilon>0$ small enough.
			
			Then for $\phi\in H^1(\mathbb{R}^n)$ it follows by Green's identity
			
			\begin{align*}
				&\int_{|x|<\rho} \left(-\Delta v+qv\right)\left(\chi\phi\right)\,dx = 0 \Leftrightarrow \int_{|x|<\rho} \nabla v\cdot\nabla(\chi\phi)+ qv(\chi\phi)\, dx = 0\\
				& \int_{\Omega} \nabla v\cdot \nabla \phi+qv\phi\,dx + \int_{B_{\rho}(0)\setminus\Omega} \nabla w\cdot\nabla(\chi\phi)\, dx =0
			\end{align*}
			
			In particular we can take $\phi=\bar{v}$ and since $w$ is given through the classical single and double layer it follows that $\nabla w \in L^2(B_{\rho}(0)\setminus\bar{\Omega})$. Thus taking the limit as $\rho\rightarrow \infty$:
			\begin{align*}
				&\int_{\Omega} |\nabla v|^2 \phi+q|v|^2\,dx + \int_{B_{\rho}(0)\setminus\Omega} \nabla w\cdot\nabla(\chi\bar{w})\, dx =0\\
				&\int_{\Omega} |\nabla v|^2 \phi+q|v|^2\,dx + \int_{B_{\rho}(0)\setminus\Omega} |\nabla w|^2\, dx = \int_{B_{\rho}(0)\setminus\Omega} \nabla w\cdot\nabla((1-\chi)\bar{w})\,dx 
				\\& \int_{\mathbb{R}^n} |\nabla v|^2+q|v|^2\,dx = 0 \Rightarrow 		\int_{\mathbb{R}^n} |\nabla v|^2+(\text{Re}\,q)|v|^2\,dx = 0	
			\end{align*}
			
			Now, we can apply Hardy's inequality for $H^1(\mathbb{R}^n)$:
			
			\begin{align*}
				\frac{(d-2)^2}{4}\int_{\mathbb{R}^n} |x|^{-2}|v|^2\,dx \leq \int_{\mathbb{R}^n} |\nabla v|^2\,dx
			\end{align*}
			
			To finally obtain the condition:
			
			\begin{align*}
				\int_{\mathbb{R}^n} \left[\frac{(d-2)^2}{4|x|^2}+\left(\text{Re }q(x)\right)\right]|v|^2\,dx \leq 0
			\end{align*}
			
			Hence, for $\text{Re q}\geq 0$ this implies that $v\equiv 0$ in $\mathbb{R}^n$. Thus $h\equiv 0$ in $\Bound$. Thus we obtain invertibility in the case $\zeta=0$.	Notice, that we have been loose on the requirement for $q$, since this will be enough for the complex-conductivity purposes, but this proof works for a larger class of potentials.
			\\
			
			Part c) follows quite easily by the fact that the set of invertible operators is open. 
			However, we present the result with the help of some estimates and Neumann series. \\

			For $h\in H^{3/2}(\Bound)$ it holds $K_{\zeta}h = S_{\zeta}(\Lambda_q-\Lambda_0)h$, due to Green's formula we have $B_{\zeta}=-\frac{1}{2}I + S_{\zeta}\Lambda_0$.

			Moreover, by Lemma \ref{harmonicestimate} and for $h\in H^{3/2}(\Bound)$ we have the decomposition $S_{\zeta}(\Lambda_q-\Lambda_0)h = S_0(\Lambda_q-\Lambda_0)h + \mathcal{H}_{\zeta}(\Lambda_q-\Lambda_0)h$. Moreover, we also have by the Lemma the estimate:
			\begin{align*}
			\|\mathcal{H}_{\zeta}(\Lambda_q-\Lambda_0)h\|_{H^{3/2}(\Bound)}\leq C|\zeta|^{n-2}\|(\Lambda_q-\Lambda_0)h\|_{H^{1/2}(\Bound)} \leq C|\zeta|^{n-2}\|h\|_{H^{3/2}(\Bound)}.
			\end{align*}
			
			From the invertibility of $I+K_0$ we obtain the decomposition					
			\begin{align*}
				[I+K_{\zeta}] = I+K_{0}+\mathcal{H}_{\zeta}\left(\Lambda_q-\Lambda_0\right) = \left(I+K_0\right)\left(I+\left(I+K_0\right)^{-1}\mathcal{H}_{\zeta}\left(\Lambda_q-\Lambda_0\right)\right)
			\end{align*}
			
			and if $$\|\left(I+K_0\right)^{-1}\mathcal{H}_{\zeta}\left(\Lambda_q-\Lambda_0\right)\|_{\mathcal{L}(H^{3/2}(\Bound))}<1$$ we obtain invertibility for $I+K_{\zeta}$ in $H^{3/2}(\Bound)$.
			
			This norm can be translated to an estimate for $\zeta$ by the above on $\mathcal{H}_{\zeta}$. We have
			
			\begin{align*}
				&\|\left(I+K_0\right)^{-1}\mathcal{H}_{\zeta}\left(\Lambda_q-\Lambda_0\right)\|_{\mathcal{L}(H^{3/2}(\Bound))} \\&\leq C|\zeta|^{n-2}\left\|\left(I+K_0\right)^{-1}\right\|_{\mathcal{L}(H^{3/2}(\Bound))}\left\|\mathcal{H}_{\zeta}\left(\Lambda_q-\Lambda_0\right)\right\|_{\mathcal{L}(H^{3/2}(\Bound))} < 1.
			\end{align*}	
			
			 Hence, for $$|\zeta|< \left[\frac{1}{\left\|\left(I+K_0\right)^{-1}\right\|_{\mathcal{L}(H^{3/2}(\Bound))}\left\|\mathcal{H}_{\zeta}\left(\Lambda_q-\Lambda_0\right)\right\|_{\mathcal{L}(H^{3/2}(\Bound))}}\right]^{1/(n-2)}=:\epsilon,$$ invertibility is obtained by Neumann series.\\

			Part iv) uses the existence of exponential growing solutions for large values of $|\zeta|$.\\
			
			Let $R>0$ be large enough such that for $\zeta\in\mathbb{C}^n$ with $\zeta\cdot\zeta=0, \, |\zeta|>R$ we have unique exponential growing solutions of (\ref{Schr_int_eq1}), corollary \ref{uniqueness} . Under this conditions, we have showed that $K_{\zeta}:=S_{\zeta}\Lambda_q-B_{\zeta}-\frac{1}{2}I$ is compact in $H^{3/2}(\Bound)$. Therefore, $I+K_{\zeta}$ is a Fredholm operator of index zero in $H^{3/2}(\Bound)$. We need to show that the kernel is empty to prove that it is invertible.
			
			Let $g\in H^{3/2}(\Bound)$ be in $\text{ker}\, K$. Then $h = [-S_{\zeta}\Lambda_q + D_{\zeta}]g$ solves the exterior problem i), ii), iv) and fulfills the radiation condition (\ref{Sommerfeld}) (the proof is analogous to Lemma \ref{BIeq}).\\
			
			Moreover, we can extend $h$ to a solution $\tilde{h}$ of $\tilde{h} = -\int_{\mathbb{R}^n} G_{\zeta}(x-y)q(y)\tilde{h}(y)\, dy$ in all of $\mathbb{R}^n$ (analogous to the previous lemma). By the estimates on $G_{\zeta}$ we note that $e^{-ix\cdot\zeta}\tilde{h}\in L^2_{\delta-1}(\mathbb{R}^n), \, 0<\delta<1$ and $$e^{-ix\cdot\zeta}\tilde{h} = -A_{\zeta}(e^{-ix\cdot\zeta}\tilde{h})$$ with $A_{\zeta}$ defined as in (\ref{K_zeta}). Since, we took $R>0$ large enough then $A_\zeta$ is a contraction in $L^2_{\delta-1}(\mathbb{R}^n)$ and this forces $\tilde{h}\equiv 0$. Therefore, 
			
			\begin{align*}
			g\equiv 0\text{ and } I+K_{\zeta} \text{ is invertible in } H^{3/2}(\Bound).
			\end{align*} 
		\end{proof}
			
			Therefore, we can solve the boundary integral equation for small and large values of $|\zeta|$ and obtain $\psi$ on $\Bound$ by:
			
			\begin{align*}
			\psi(x,\zeta) = \left[\frac{1}{2}I + S_{\zeta}\Lambda_q - B_{\zeta}\right]^{-1}\left(e^{ix\cdot\zeta}\right)			\end{align*}
			 	
		This allows us to obtain the scattering transform from the boundary data:
		
		\begin{theorem}
			Suppose that $\Omega$ is a bounded $C^{1,1}$-domain in $\mathbb{R}^n,\,n\geq 3$. Let $q\in L^{\infty}(\Omega)$ be complex-valued and suppose that $0$ is not a Dirichlet eigenvalue of $-\Delta+q$ in $\Omega$.
			
			We define the scattering transform for non-exceptional points $\zeta\in\mathcal{V}$ by			
			\begin{align}\label{scat}
			\textnormal{\textbf{t}}(\xi, \zeta) = \int_{\mathbb{R}^3} e^{-ix\cdot(\zeta+\xi)}q(x)\psi(x,\zeta)\,dx,\; \xi\in\mathbb{R}^n.
			\end{align}
			
			Then, for each $\xi\in\mathbb{R}^n$ we can compute the scattering transform for the non-exceptional points $\zeta\in\mathcal{V}_{\xi} :=\left\{\zeta\in\mathbb{C}^n\setminus\{0\}: \zeta\cdot\zeta=0,\, |\xi|^2+2\zeta\cdot\xi=0\right\}$ from the solutions of the boundary integral equation by:
			
			\begin{align}\label{boundaryscat}
			\textnormal{\textbf{t}}(\xi,\zeta) = \int_{\Bound} e^{-ix\cdot(\zeta+\xi)}\left[\Lambda_q + i(\xi+\zeta)\cdot\nu\right]\psi(x,\zeta)\,ds(x), \; \xi\in\mathbb{R}^n.
			\end{align}

		\end{theorem}
		
		\begin{proof}
			
			From the Lemma \ref{Bitch2} and \ref{Bitch} we obtain unique exponentially growing solutions of (\ref{Schr_int_eq1}) by the one-to-one relation with the boundary integral (\ref{boundaryInteq}).
			Therefore, by Green identity it holds:
			
			\begin{align*}
			\textnormal{\textbf{t}}(\xi,\zeta)& =\int_{\Omega} e^{-ix\cdot(\xi+\zeta)}q(x)\psi(x,\zeta)\,dx \\
			&= \int_{\Omega} e^{-ix\cdot(\xi+\zeta)}\Delta\psi(x,\zeta) - \left(\Delta e^{-ix\cdot(\zeta+\xi)}\right)\psi(x,\zeta)\,dx \\
			&= \int_{\Bound} e^{-ix\cdot(\xi+\zeta)}\left[\Lambda_q\psi(x,\zeta) + i(\xi+\zeta)\cdot\nu \psi(x,\zeta)\right]\,ds(x)
			\\
			& = \int_{\Bound} e^{-ix\cdot(\xi+\zeta)}\left[\Lambda_q+ i(\xi+\zeta)\cdot\nu \right]\psi(x,\zeta)\,ds(x)
			\end{align*}
			
			for $\xi\in\mathbb{R}^n$ and $\zeta\in\mathcal{V}_{\xi}$ such that the boundary integral equation has a unique solution.
			
		\end{proof}
		
		\section{From \textbf{t} to $\gamma$}
		
		From the scattering transform we can obtain the Fourier transform of the potential through large asymptotics. Unfortunately for this we need to solve the boundary integral equation for large values $\zeta$, which is makes this method very unstable. In \cite{Hamilton_3D} they avoid the boundary integral equation by using the approximation $\psi(x,\zeta)\approx e^{ix\cdot\zeta}$ to compute the scattering transform. This simplified version was even applied for complex conductivities in order to obtain a stable reconstruction procedure. 
		
		This method is based on the following asymptotic:
		
		\begin{theorem}
			Let $\Omega$ be a bounded $C^{1,1}$-domain in $\mathbb{R}^n,\,n\geq 3$.
			Let $q\in L^{\infty}(\Omega)$ be a complex-valued potential extended to zero outside $\Omega$, such that $0$ is not a Dirichlet eigenvalue of $(-\Delta+q)$.
			Then for $|\zeta|>R$ and $0<\delta<1$:
			\begin{align}
			|\textnormal{\textbf{t}}(\xi,\zeta)-\hat{q}(\xi)|\leq \frac{\tilde{c}(\delta, R)}{|\zeta|} \|q\|^2_{\delta}, \quad \forall\xi\in\mathbb{R}^n
			\end{align}
		\end{theorem}
		
		\begin{proof}
			The proof follows trivially by the corollary \ref{uniqueness}. If $q\in L^{\infty}(\Omega)$ is a complex-valued and compactly supported potential it follows that $\hat{q}$ is well-defined and
			\begin{align*}
			\left|\textbf{t}(\xi,\zeta)-\hat{q}(\xi)\right|&=\left|\int e^{-ix\cdot\xi}q(x)\left[e^{-ix\cdot\zeta}\psi(x,\zeta)-1\right]\, dx\right| \\
			&\leq \|q\|_{1-\delta}\|e^{-ix\cdot\zeta}\psi(x,\zeta)-1\|_{\delta-1} \leq \frac{\tilde{c}(\delta, R)}{|\zeta|}\|q\|_{\infty}^2.
			\end{align*}			
		\end{proof}
		
		Following the $\bar{\partial}$ compatibility equations satisfied by $\textbf{t}$ known from (\cite{Nachman_Ablowitz}, \cite{Beals_Coifman}, \cite{Henkin_Novikov}), Nachman was able to derive a formula to calculate $\hat{q}$ when we know $\textbf{t}(\xi, \zeta)$ for $\xi\in\mathbb{R}^n$, $|\zeta|\geq M$, $(\xi+\zeta)^2=0$ and without requiring taking the limit of $|\zeta|\rightarrow\infty$. We can follow his method directly to obtain a formula for the complex-potential. \\

		For such, let $\psi(x,\zeta)$ be the solution of (\ref{Schr_int_eq1}) with $e^{-ix\cdot\zeta}\psi(x,\zeta)-1\in L^2_{\delta-1}(\mathbb{R}^n)$, that is, $\zeta$ is not an exceptional point.
		
		Define,
		\begin{align}
			\mu(x,\zeta):= |q(x)|e^{-ix\cdot\zeta}\psi(x,\zeta)
		\end{align}
		then $\mu$ solves the following integral equation
		
		\begin{align}\label{mu_eq}
			\mu(x,\zeta) = |q(x)| - |q(x)|\int_{\mathbb{R}^n} g_{\zeta}(x-y)\tilde{q}(y)\mu(y,\zeta)\,dy
		\end{align}
		
		Hereby, we set $\tilde{A}_{\zeta}f(x):= |q(x)|\int_{\mathbb{R}^n} g_{\zeta}(x-y)\tilde{q}(y)f(y)\,dy$ with $\tilde{q}(x)=q(x)/|q(x)|$ in the support of $q$ and $0$ otherwise. Moreover, the scattering transform is given through:
		
		\begin{align}
			\textbf{t}(\xi,\zeta):= \int_{\mathbb{R}^n} e^{-ix\cdot\xi}\tilde{q}(y)\mu(x,\zeta)\,dx.
		\end{align}
		
		\begin{lemma}
			Suppose $q\in L^{\infty}(\mathbb{R}^n)$ with compact support. Let $R>c(\delta, a)\|q(x)\langle x\rangle\|_{L^\infty}$ with $\delta\in(0,1)$ and $c(\delta,a)$ as in proposition \ref{conv_est}.
			\begin{enumerate}[label=(\alph*)]
				\item If $\zeta\geq R,\, \zeta\cdot\zeta=0$, then (\ref{mu_eq}) has a unique solution $\mu(\cdot,\zeta)$ in $L^2(\mathbb{R}^n)$ with compact support.
				\item For $|\zeta|>M$, $\zeta\cdot\zeta=0$ and all $w\in\mathbb{C}^n$ with $w\cdot\bar{\zeta}=0$,
				\begin{align}
					w\cdot\frac{\partial\mu}{\partial\bar{\zeta}}(x,\zeta) = \frac{-1}{(2\pi)^{n-1}}\int e^{ix\cdot\xi}w\cdot\xi\delta(|\xi|^2+2\zeta\cdot\xi)t(\xi, \zeta)\mu(x,\zeta+\xi)\,d\xi.
				\end{align}
			\end{enumerate}
		\end{lemma}
	
		\begin{proof}
			For details check Nachman \cite{Nachman}.
		\end{proof}
	
		We keep it short here and refer to Nachman \cite{Nachman} for the formula to obtain $q$ without taking limits of the scattering transform.\\
		
		Our interest resides now in the behavior of exponential growing solutions for $\zeta$ close to zero. Due to invertibility of the boundary integral equation we can in fact show that there are no exceptional points near $0$. Therefore, analogously to \cite{Knudsen_3D} we are able to obtain the following estimate:
		
		\begin{lemma}
			
			Let $\gamma\in C^{1,1}(\Omega)$ be the complex-conductivity with $\sigma\geq c>0, \, \epsilon\geq0,\, \omega\in\mathbb{R^+}$ and suppose $\gamma\equiv 1$ near $\Bound$. Set $q=(\Delta\gamma^{1/2})/\gamma^{1/2} \in L^{\infty}(\Omega)$. 
			
			For $\zeta\in\mathcal{V}$ sufficiently small  and $\phi\in H^{3/2}(\Bound)$ the corresponding boundary integral solution of (\ref{boundaryInteq}), it holds 
			\begin{align}
			\|\phi(\cdot,\zeta)-1\|_{H^{3/2}(\Bound)} \leq C|\zeta|
			\end{align}
			
		\end{lemma}

		\begin{proof}
			Let $K_{\zeta} = S_{\zeta}\left(\Lambda_q-\Lambda_0\right)$. 
			Solutions of the boundary integral equation fulfill:
			$$\phi(x,\zeta) - 1 = \left(e^{ix\cdot\zeta}-1\right)-K_{\zeta}\left(\phi(x,\zeta)-1\right),$$
			
			which follows by $(\Lambda_q1)=0,\, (\Lambda_0)1=0$, since the unique $H^2$-solution of $(-\Delta+q)u=0,\, \left.u\right|_{\Bound}=1$ is $\gamma^{1/2}$ and $w=1$ is the unique harmonic function in $H^2(\Omega)$ with boundary value $1$.
			
			Under the conditions on $\gamma$ it holds that $\text{Re }q>0$ and hence by proposition \ref{InvertibilityBIE} it holds that $[I+K_{\zeta}]$ is invertible in $H^{3/2}(\Bound)$ for small $\zeta$ and hence, $$\phi-1=\left[I+K_{\zeta}\right]^{-1}\left(e^{ix\cdot\zeta}-1\right).$$
			
			It clearly holds that by Taylor series that $\|e^{ix\cdot\zeta}-1\|_{H^{3/2}(\Bound)}\leq C_1|\zeta|$ and $\left\|[I+K_{\zeta}]^{-1}\right\|_{\mathcal{L}(H^{3/2}(\Bound))}$ is uniformly bounded for small $|\zeta|$ due to Neumann series inversion.
			
			Hence, 
			\begin{align*}
			\|\phi-1\|_{H^{3/2}(\Bound)} \leq C_2\|e^{ix\cdot\zeta}-1\|_{H^{3/2}(\Bound)}\leq C_3|\zeta|
			\end{align*}
		\end{proof}
	
		Hence, it follows:
	
		\begin{theorem}
			Let $\gamma\in C^{1,1}(\Omega)$ be the complex-conductivity with $\sigma\geq c>0, \, \epsilon\geq0, \omega\in\mathbb{R^+}$ and suppose $\gamma\equiv 1$ near $\Bound$. Set $q=(\Delta\gamma^{1/2})/\gamma^{1/2} \in L^{\infty}(\Omega)$. 
			
			For $\zeta\in\mathcal{V}$ small enough such that (\ref{Schr_int_eq1}) has unique exponentially growing solutions $\psi(x,\zeta)$, it holds:
			\begin{align}
			\|\psi(\cdot,\zeta)-\gamma^{1/2}(\cdot)\|_{L^2(\Omega)}\leq C|\zeta|
			\end{align} 
		\end{theorem}
		
		\begin{proof}
			Since $\gamma=1$ near the boundary $\Bound$ we have that $\gamma^{1/2}$ is the unique $H^2(\Omega)$ solution of 
			\begin{align*}
			\begin{cases}
			-\Delta u+qu=0, \text{ in } \Omega\\
			\left.u\right|_{\Bound}=1.
			\end{cases}
			\end{align*}

			By the elliptic estimates, we obtain that
			
			$$\|\psi(\cdot,\zeta)-\gamma^{1/2}(\cdot)\|_{L^2(\Omega)} \leq \|\psi(\cdot,\zeta)-\gamma^{1/2}(\cdot)\|_{H^2(\Omega)} \leq \|\psi(\cdot,\zeta)-\gamma^{1/2}(\cdot)\|_{H^{3/2}(\Bound)}  \leq C|\zeta|.$$\\
		\end{proof}
		
		This theorem states that we can reconstruct the complex-conductivity from the exponential growing solutions by:		
		\begin{align}
			\gamma(x) = \lim_{|\zeta|\rightarrow 0} \psi(x,\zeta),\quad \text{for a.e. } x\in\Omega.
		\end{align}
	
		However, recall that for small $\zeta$ we only know how to obtain the boundary values of the exponential growing solutions from the boundary measurements. To provide a reconstruction of $\gamma$ in $\Omega$ it is necessary to compute these solutions for all $\zeta$ small inside $\Omega$ from the scattering data. This might be possible by the $\bar{\partial}$-equation.\\  
		
		In order to obtain a $\bar{\partial}$ reconstruction method complex conductivities the following problems need to be solved:
				
		\begin{enumerate}
			\item \textbf{Uniqueness of (\ref{mu_eq}) for $\zeta$ non-exceptional and considerably small:} A first step is to show that this equation is uniquely solvable for small values of $\zeta$. In Nachman's proof invertibility of the operator $I+\tilde{A}_{\zeta}$ in $L^2(\mathbb{R}^n)$ follows by that of $I+A_{\zeta}$ obtained in corollary \ref{uniqueness} for $\zeta$ large. In the case of small values, we shows that the integral equation (\ref{Schr_int_eq1}) is uniquely solvable by the unique solvability of the boundary integral equation. However, this will not imply that the operator $I+A_{\zeta}$ is invertible and therefore it does not to the conclusion desired here. Hence, a different proof for the existence of a unique solution $\mu$ to (\ref{mu_eq}). \\
			
			\item \textbf{Solvability of $\bar{\partial}-$equation:} In \cite{Nachman} there is no proof that $\bar{\partial}-$equation is uniquely solvable, but this is essential since this would be the only equation fully independent of $q$ and where its information is given through the scattering transform. In this sense, we need to study the equation in the space $\mathcal{V}\setminus\{\zeta\in\mathbb{C}^n: \epsilon\leq |\zeta|< R\}$. In the work of \cite{Lakshtanov_Vainberg} they establish this approach in a two-dimensional positive energy setting and intuition could lead to a similar work in our case.\\
		\end{enumerate}

		\section{Reconstruction of $\Lambda_q$ from the boundary measurements $\Lambda_{\gamma}$}
		
		The Dirichlet-to-Neumann map $\Lambda_{\gamma}$ is bounded from $H^{1/2}(\Bound)$ to $H^{-1/2}(\Bound)$. Moreover, it is properly defined through:
		
		\begin{align}
		\langle \Lambda_{\gamma}f, g\rangle = \int_{\Omega} \gamma\nabla u\cdot\nabla v\, dx,
		\end{align}
		
		where $u$ is the unique $H^1(\Omega)$ solution of the interior problem $\nabla\cdot(\gamma\nabla u)=0$ in $\Omega$ and $\left. u \right|_{\Bound}=f$ and $v\in H^{1}(\Omega)$ with $\left. v\right|_{\Bound}=g$.
		\\
		
		We can also define the Dirichlet-to-Neumann map for the Schrödinger operator by:
		
		\begin{align*}
		&\Lambda_q: H^{1/2}(\Bound) \rightarrow H^{-1/2}(\Bound) \\
		&\langle\Lambda_q\tilde{f},\tilde{g}\rangle = \int_{\Omega} \nabla\tilde{u}\cdot\nabla\tilde{v}+q\tilde{u}\tilde{v}\,dx,\quad \forall \tilde{v}\in H^1(\Omega),\text{ s.t. } \left.\tilde{v}\right|_{\Bound} = \tilde{g},
		\end{align*}
		
		and $\tilde{u}\in H^1(\Omega)$ is the unique solution to $(-\Delta+q)\tilde{u}=0, \text{ in } \Omega,\; \left.\tilde{u}\right|_{\Bound}=\tilde{f}$.
		\\
		
		As in the real case, since both problems are interconnected we can obtain $\Lambda_q$ from $\Lambda_{\gamma}$ by:
		
		\begin{align}
		\Lambda_q = \gamma^{-1/2}\left[\Lambda_{\gamma} +\frac{1}{2}\frac{\partial\gamma}{\partial\nu}\right]\gamma^{-1/2}.\\
		\nonumber
		\end{align}
		
		This brings to light that we can determine $\Lambda_q$ from $\Lambda_{\gamma}$ and the boundary values $\left.\gamma\right|_{\Bound}$ and $\left.\frac{\partial\gamma}{\partial\nu}\right|_{\Bound}$. Thus, if $\gamma\equiv 1$ near $\Bound$ then for $\gamma\in W^{2,\infty}(\Omega)$ it holds that $\Lambda_q=\Lambda_{\gamma}$. Otherwise, we need to obtain a method to reconstruct these boundary values.
		\\
		
		There are many results to compute these boundary values. However, most of them need further smoothness. Still Nachman holds the best result for our case. In \cite{Nachman2D} he showed that the boundary values can be obtained without further smoothness assumptions. Following his proof we see that there is no explicit requirement of $\gamma$ being real, besides the fact that $\gamma\geq c>0$ and uniqueness of the Dirichlet problem in $H^1(\Omega)$. Hence, we can quickly extend the result for complex-conductivities in $W^{2,\infty}(\Omega)$ with $\text{Re }\gamma\geq c>0$. 
		
		The result is obtained through the following lemmas:
		
		\begin{lemma}
			Let $\Omega$ be a bounded $C^{1,1}$-domain in $\mathbb{R}^n,\, n\geq 2$. Assume $\gamma\in W^{1,r}(\Omega)$ for $r>n$ and $\text{Re }\gamma\geq c>0$.
			
			Then for any $f\in H^{1/2}(\Bound)$ and $$h\in \mathring{H}^{-1/2}(\Bound):=\left\{h\in H^{-1/2}(\Bound): <h,1>_{\Bound_j}=0,\, j=1,...,N \right\}$$ the identity holds:			
			\begin{align}
			\langle h, (\gamma-\mathcal{R}\Lambda_{\gamma})f\rangle = \int_{\Omega} u\nabla w\cdot\nabla \gamma,
			\end{align}
		where $u\in H^1(\Omega) \text{ solution of } \nabla\cdot(\gamma \nabla u)=0,\, \left. u\right|_{\Bound} = f,$ and $w\in H^1(\Omega)$ is a weak solution of $\Delta w=0$ in $\Omega$ with $\frac{\partial w}{\partial\nu}=h$ and $\mathcal{R}$ denotes the Neumann-to-Dirichlet map.			
		\end{lemma}
		
		\begin{lemma}
			Let $\Omega$ be a bounded $C^{1,1}$-domain in $\mathbb{R}^n, n\geq 2$. Assume $\gamma$ is in $W^{2,p}(\Omega)$, $p>n/2$ and $\text{Re }\gamma\geq c>0$.
			
			For any $f,\,g\in H^{1/2}(\Bound)$ the identity holds:
			\begin{align*}
			\langle g, \left(2\Lambda_{\gamma}-\Lambda_1\gamma-\gamma\Lambda_1+\frac{\partial\gamma}{\partial\nu}\right)f\rangle = \int_{\Omega} 2v\nabla(u-u_0)\cdot\nabla\gamma+v(2u-u_0)\Delta\gamma\,dx
			\end{align*}
			where $u, u_0, v$ are respectively the $H^1(\Omega)$ solutions of $L_{\gamma}(u)=0, \, \Delta u_0=0\text{ and } \Delta v=0,$ in $\Omega$, with $\left.u\right|_{\Bound}=f, \left.u_0\right|_{\Bound}=f$ and $\left.v\right|_{\Bound}=g$.
		\end{lemma}
		
		From this we obtain the boundary reconstruction formulas.
		
		\begin{theorem}
			Let $\Omega$ be a bounded $C^{1,1}$-domain in $\mathbb{R}^n, n\geq 2$. Suppose $\gamma\in W^{1,r}(\Omega), \,r>n$ and $\text{Re }\gamma\geq c>0$.
			\begin{enumerate}[label=(\roman*)]
				\item $\left.\gamma\right|_{\Bound\cap U}$ can be recovered from $\Lambda_{\gamma}$ by:
				\begin{align}
				\langle h, \gamma f\rangle = \lim_{\substack{|\eta|\rightarrow \infty\\ \eta\in\mathbb{R}^{n-1}\times\{0\}}} \langle h_{\eta}, \mathcal{R}\Lambda_{\gamma} e^{-i\langle\cdot,\eta\rangle }f\rangle,
				\end{align}
				with $f\in H^{1/2}(\Bound)\cap C(\Bound)$ and $h\in L^2(\Omega)$ supported in $U\cap\Bound$ and $h_{\eta}$ is defined as zero outside $\Bound\cap U$ and:
				$$ h_{\eta}(x) = h(x)e^{-ix\cdot\eta} - \frac{1}{|\Bound\cap U|} \int_{\Bound\cap U} h(y)e^{-iy\cdot\eta}\, dy,\; \text{ for } x\in\Bound\cap U.$$
				
				\item If $\gamma\in W^{2,r}, \,r>n/2$, then for any continuous function $f, g$ in $H^{1/2}(\Bound)$ with support in $\Bound\cap \Bound$ holds
				
				\begin{align}
				\langle g, \frac{\partial\gamma}{\partial\nu}f\rangle = \lim_{\substack{|\eta|\rightarrow \infty\\ \eta\in\mathbb{R}^{n-1}\times\{0\}}} \langle g, e^{-i\langle \cdot,\eta\rangle}\left(\gamma\Lambda_1+\Lambda_1\gamma-2\Lambda_{\gamma}\right)e^{i\langle\cdot,\eta\rangle}f\rangle.\\
				\nonumber
				\end{align}
			\end{enumerate}
		\end{theorem}

		\textit{ Acknowledgements.} Ivan Pombo was supported by the Portuguese funds through the CIDMA-Center for Research and Development in Mathematics and Applications and the Portuguese Foundation for Science and Technology (Fundação para a Ciência e a Tecnologia [FCT]), within Projects UIDB/04106/2020 and UIDP/04106/2020, and by the FCT PhD Grant SFRH/BD/143523/2019.


\begin{thebibliography}{102}
			
				
			\bibitem{Alessandrini}
			Alessandrini, G. (1988). Stable determination of conductivity by boundary measurements. Applicable Analysis, 27(1-3), 153-172.
			
			\bibitem{Astala_Faraco_Rodgers}
			Astala, K., Faraco, D., Rogers, K. M. (2013). Unbounded potential recovery in the plane. Ann. Sci. Éc. Norm. Supér.(4).
			
			\bibitem{Astala_Paivarinta}
			Astala, K., Päivärinta, L. (2006). Calderón's inverse conductivity problem in the plane. Annals of Mathematics, 265-299.
			
			\bibitem{Beals_Coifman}
			Beals, R. (1985). Multidimensional inverse scatterings and nonlinear partial differential equations. In Proc. Symp. Pure Math. (Vol. 43, pp. 45-70).
			
			\bibitem{Imanuvilov_Yamamoto}
			Blåsten, E., Imanuvilov, O. Y., Yamamoto, M. (2015). Stability and uniqueness for a two-dimensional inverse boundary value problem for less regular potentials. Inverse Problems \& Imaging,  9(3), 709-723. 
			
			\bibitem{Borcea}
			Borcea, L. (2002). Electrical impedance tomography. Inverse problems, 18(6), R99.
			
			\bibitem{Brown}
			Brown, R. M. (1996). Global uniqueness in the impedance-imaging problem for less regular conductivities. SIAM Journal on Mathematical Analysis, 27(4), 1049-1056.
			
			\bibitem{Brown_Torres}
			Brown, R. M., Torres, R. H. (2003). Uniqueness in the inverse conductivity problem for conductivities with 3/2 derivatives in $L^p, p>2n$. Journal of Fourier Analysis and Applications, 9(6), 563-574.
			
			\bibitem{Brown_Uhlmann}
			Brown, R. M., Uhlmann, G. A. (1997). Uniqueness in the inverse conductivity problem for nonsmooth conductivities in two dimensions. Communications in partial differential equations, 22(5-6), 1009-1027.
			
						
			\bibitem{Bukgheim}
			Bukhgeim, A. L. (2008). Recovering a potential from Cauchy data in the two-dimensional case.
			
			
			\bibitem{Calderon}
			Calderón, A. P. (2006). On an inverse boundary value problem. Computational \& Applied Mathematics, 25, 133-138.
			
			\bibitem{Chanillo}
			Chanillo, S. (1990). A problem in electrical prospection and an n-dimensional Borg-Levinson theorem. Proceedings of the American Mathematical Society, 108(3), 761-767.
			
			\bibitem{Knudsen_3D}
			Cornean, H., Knudsen, K., Siltanen, S. (2006). Towards a d-bar reconstruction method for three-dimensional EIT.
			
			\bibitem{Faddeev}
			Faddeev, L. D. (1965). Growing solutions of the Schrödinger equation, Dokl. Akad. Nauk SSSR, 165, 514–517.
			
			\bibitem{Francini}
			Francini, E. (2000). Recovering a complex coefficient in a planar domain from the Dirichlet-to-Neumann map. Inverse Problems, 16(1), 107.
			
			\bibitem{Gilbarg_Trudinger}
			Gilbarg, D., Trudinger, N. S. (2015). Elliptic partial differential equations of second order (Vol. 224). springer.
			
			\bibitem{Hamilton_3D}
			Hamilton, S. J., Isaacson, D., Kolehmainen, V., Muller, P. A., Toivainen, J., Bray, P. F. (2021). 3D Electrical Impedance Tomography reconstructions from simulated electrode data using direct inversion $\mathbf {t}^{\rm {{\textbf {exp}}}} $ and Calderón methods. Inverse Problems \& Imaging.
			
			\bibitem{Henkin_Novikov}
			Henkin, G. M., Novikov, R. G. (1988). A multidimensional inverse problem in quantum and acoustic scattering. Inverse problems, 4(1), 103.

			\bibitem{Knudsen_Tamasan}
			Knudsen, K., Tamasan, A. (2003). Reconstruction of less regular conductivities in the plane. Communications in Partial Differential Equations, 1, 28.
			
			\bibitem{Lakshtanov_Vainberg}
			Lakshtanov, E., Vainberg, B. (2017). Recovery of Lp-potential in the plane. Journal of Inverse and Ill-posed Problems, 25(5), 633-651.
			
			\bibitem{Lakshtanov_Tejero_Vainberg}
			Lakshtanov, E., Tejero, J., Vainberg, B. (2017). Uniqueness in the inverse conductivity problem for complex-valued Lipschitz conductivities in the plane. SIAM Journal on Mathematical Analysis, 49(5), 3766-3775.
			
			\bibitem{McLean}
			McLean, W. (2000). Strongly elliptic systems and boundary integral equations. Cambridge university press.
			
			\bibitem{Nachman_Ablowitz}
			Nachman, A. I., Ablowitz, M. J. (1984). A Multidimensional Inverse‐Scattering Method. Studies in applied mathematics, 71(3), 243-250.
			
			\bibitem{Nachman}
			Nachman, A. I. (1988). Reconstructions from boundary measurements. Annals of Mathematics, 128(3), 531-576.
			
			\bibitem{Nachman2D}
			Nachman, A. I. (1996). Global uniqueness for a two-dimensional inverse boundary value problem. Annals of Mathematics, 71-96.
						
			\bibitem{Nach_Syl_Uhl}
			Nachman, A., Sylvester, J., Uhlmann, G. (1988). An n-dimensional Borg-Levinson theorem. Communications in Mathematical Physics, 115(4), 595-605.
			
			\bibitem{Novikov}
			Novikov, R. G. (1988). Multidimensional inverse spectral problem for the equation $-\Delta\psi+(v(x) - Eu (x))\psi= 0$. Functional Analysis and Its Applications, 22(4), 263-272.
			
			\bibitem{Novikov_Santacesaria}
			Novikov, R. G., Santacesaria, M. (2011). Global uniqueness and reconstruction for the multi-channel Gelfand–Calderón inverse problem in two dimensions. Bulletin des Sciences Mathematiques, 135(5), 421-434.
			
			\bibitem{PPU}
			Päivärinta, L., Panchenko, A., Uhlmann, G. (2003). Complex geometrical optics solutions for Lipschitz conductivities. Revista Matematica Iberoamericana, 19(1), 57-72.
			
			\bibitem{Pombo}
			Pombo, I. (2020). CGO-Faddeev approach for complex conductivities with regular jumps in two dimensions. Inverse Problems, 36(2), 024002.

			\bibitem{Sylvester_Uhlmann}
			Sylvester, J., Uhlmann, G. (1987). A global uniqueness theorem for an inverse boundary value problem. Annals of mathematics, 153-169
			
			
		\end{thebibliography}
	\end{document}